\documentclass[a4paper,11pt,reqno]{amsart}

\usepackage{enumerate}
\usepackage{amssymb, amsmath}
\usepackage{mathrsfs}
\usepackage{amscd}
\usepackage{verbatim}

\renewcommand{\emptyset}{\varnothing}

\theoremstyle{plain}
\newtheorem{theorem}{Theorem}[section]
\newtheorem{corollary}[theorem]{Corollary}
\newtheorem{lemma}[theorem]{Lemma}
\newtheorem{proposition}[theorem]{Proposition}

\theoremstyle{remark}
\newtheorem{remark}[theorem]{Remark}
\newtheorem{example}[theorem]{Example}

\numberwithin{equation}{section}

\def\R{{\mathbb R}}

\newcommand{\F}{{\mathcal F}}

\renewcommand{\S}{\Sigma}
\newcommand{\eps}{\varepsilon}

\DeclareMathOperator{\tr}{tr}

\newcommand{\beq}{\begin{equation}}
\newcommand{\eeq}{\end{equation}}
\newcommand{\bal}{\begin{aligned}}
\newcommand{\eal}{\end{aligned}}
\newcommand{\ben}{\begin{enumerate}}
\newcommand{\beni} {\begin{enumerate}[(i)]}
\newcommand{\een}{\end{enumerate}}
\newcommand{\bit}{\begin{itemize}}
\newcommand{\eit}{\end{itemize}}
\newcommand{\beqw}{\begin{equation*}}
\newcommand{\eeqw}{\end{equation*}}
\newcommand{\bthm}{\begin{theorem}}
\newcommand{\ethm}{\end{theorem}}
\newcommand{\bpr}{\begin{proposition}}
\newcommand{\epr}{\end{proposition}}
\newcommand{\ble}{\begin{lemma}}
\newcommand{\ele}{\end{lemma}}
\newcommand{\blem}{\begin{lemma}}
\newcommand{\elem}{\end{lemma}}
\newcommand{\bpf}{\begin{proof}}
\newcommand{\epf}{\end{proof}}
\newcommand{\bex}{\begin{example}}
\newcommand{\eex}{\end{example}}
\newcommand{\bre}{\begin{example}}
\newcommand{\ere}{\end{example}}
\newcommand{\bma}{\begin{bmatrix}}
\newcommand{\ema}{\end{bmatrix}}

\newcommand{\Dom}{{\mathsf D}}

\newcommand{\one}{{{\bf 1}}}

\newcommand{\ip}[1]{\langle {#1}\rangle}

\newcommand{\ov}{\overline}
\newcommand{\ot}{\otimes}

\newcommand{\G}{\Gamma}
\renewcommand{\phi}{\varphi}

 \newcommand{\calP}{\mathscr{P}}
 \renewcommand{\H}{\mathcal{H}}
 \newcommand{\V}{\mathcal{V}}
 \renewcommand{\hat}{\widehat}

  \renewcommand{\r}{\mathbf{r}}
 \newcommand{\ux}{\underline{x}}
 \newcommand{\ox}{\overline{x}}
 \newcommand{\uxi}{\underline{x^i}}
 \newcommand{\oxi}{\overline{x^i}}
 \newcommand{\uy}{\underline{y}}
 \newcommand{\oy}{\overline{y}}
 \newcommand{\ut}{\underline{k}}
 \renewcommand{\ot}{\overline{k}}
 \newcommand{\us}{\underline{j}}
 \newcommand{\os}{\overline{j}}

 \newcommand{\thh}{{\widetilde{h}}}
 \newcommand{\h}{\mathbf{h}}
 \newcommand{\hx}{{\widehat{x}}}

\begin{document}

\title[A Trotter product formula for gradient flows]
{A Trotter product formula for gradient flows in metric spaces}

\author{Philippe Cl\'ement}
\address{
Delft Institute of Applied Mathematics\\
Delft University of Technology\\ P.O. Box 5031\\ 2600 GA
Delft\\The Netherlands} 
\email{ppjeclement@aim.com}

\author{Jan Maas}
\address{
Institute for Applied Mathematics\\
University of Bonn\\
Endenicher Allee 60\\
53115 Bonn\\
Germany}
\email{maas@iam.uni-bonn.de}

\thanks{The second named author is supported by Rubicon subsidy 680-50-0901 of the Netherlands Organisation for Scientific Research (NWO)}



 \begin{abstract}
We prove a Trotter product formula for gradient flows in metric
spaces. This result is applied to establish convergence in the $L^2$-Wasserstein metric of the splitting method for some Fokker-Planck equations and porous medium type equations perturbed by a potential.
 \end{abstract}

\maketitle
 \section{Introduction and statement of the main results} \label{sec:intro}

In the first part of \cite{AGS} Ambrosio, Gigli and Savar{\'e} developed a rich theory of gradient flows in metric spaces. In particular they studied in great detail the following situation.

Let $(X,d)$ be a complete metric space, and let $\phi : X \to \R \cup \{ + \infty\}$ be a lower semicontinuous functional which is not identically $+ \infty.$
Associated with $\phi$ is the \emph{Moreau-Yosida functional} $\Phi$ defined for $h > 0$ and $x \in X$ by
\begin{align*}
  \Phi(h,x;y) :=
   \left\{ \begin{array}{ll}
 \phi(y) + \frac{1}{2h} d^2(x,y),
 & \text{$y \in \Dom(\phi)$},\\
 + \infty,
 & \text{otherwise,}
 \end{array} \right.
\end{align*}
where $\Dom(\phi) := \{ x\in X : \phi(x)  < \infty \}.$
In \cite{AGS} conditions are given which guarantee
\begin{enumerate}
\item
existence and uniqueness of a global minimizer of $\Phi(h,x; \cdot),$ which is denoted by $J_h x$ and called the \emph{resolvent} of $\phi$ at $x$;
\item
convergence of sequences of iterated resolvents $\{ (J_{t/n})^n x \}_{n \geq 1}$;
\item
the validity of a certain \emph{evolution variational inequality} (EVI) for the limit.
\end{enumerate}

In the second part of \cite{AGS} the theory is applied to problems in the space of probability measures where the functional $\phi$ can be naturally written as the sum of two (or more) functionals $\phi^i,$ $i = 1,2.$ It appears that in most cases one can associate with each $\phi^i$ a resolvent $J_h^i$. It is therefore natural to consider the problem of convergence of sequences of iterates of the form $\{ \big(J_{t/n}^2 J_{t/n}^1\big)^n  x\}_{n \geq 1}$ provided they are well-defined and to investigate whether the limit satisfies the EVI associated with $\phi.$
In this paper we give sufficient conditions for this to be true (Theorem \ref{thm:trotter-formula}).

We apply our abstract results to establish convergence with respect to the $L^2$-Wasserstein metric of the splitting method for Fokker-Planck equations and porous medium equations with a potential satisfying appropriate conditions.

\medskip

Let us now present the setting of the paper and state the main results.
Throughout the paper, we let $(X,d)$ be a complete metric space.
For $i = 1,2,$ let $\phi^i : X \to \R \cup \{ + \infty \}$ be a lower semicontinuous (lsc) functional satisfying
\begin{align*}
 \Dom(\phi^1) \cap \Dom(\phi^2) \neq \emptyset.
\end{align*}
We consider the functional $\phi := \phi^1 + \phi^2$ defined by
\begin{align*}
 \Dom(\phi) &:=  \Dom(\phi^1) \cap \Dom(\phi^2), \\
      \phi(x) &:=  \phi^1(x) + \phi^2(x), \quad x\in \Dom(\phi),
\end{align*}
and note that $\Dom(\phi) \neq \emptyset,$ and $\phi$ is lower semicontinuous.

We shall impose three assumptions:

\bit
\item[($A_1$)] For $i = 1,2,$ for any $h > 0$ and any $x\in \overline{\Dom(\phi^i)}$, the following variational inequality has a solution:
\bit
\item[] \emph{ find $y\in \Dom(\phi^i)$ satisfying
\begin{equation}\label{eq:DEVI}
\frac{1}{2h}[d^2(y,z)-d^2(x,z)]+\frac{1}{2h}d^2(y,x)+\phi^i(y)
\leq \phi^i(z)
\end{equation}
for all $z\in \Dom(\phi^i).$}
\eit
\eit

Clearly, if $y\in \Dom(\phi^i)$ satisfies \eqref{eq:DEVI}, then $y$ is a global minimizer of $\Phi^i(h,x;\cdot)$. Since
$\frac1{2h} d^2(y,z)+\Phi^i(h,x;y)\leq \Phi^i(h,x;z)$
for every $z\in \Dom(\phi^i)$, this global minimizer is unique. We will denote the minimizer by $J_h^i x.$ Notice that for $x \in \Dom(\phi)$ and $h > 0$ we have
\begin{align*}
 \phi^i(J_h^i x) \leq \phi(x),
\end{align*}
as can be seen by setting $z = x$ in \eqref{eq:DEVI}.

\bit \item[($A_2$)]
For any $h > 0$ we have
  \begin{align*}(i)\qquad
 J_h^1\big(\ov{\Dom(\phi^1)}\cap\Dom(\phi^2) \big)
        & \subseteq\ov{\Dom(\phi^2)},\\
        (ii)\qquad
 J_h^2 \big(\Dom(\phi^1)\cap\ov{\Dom(\phi^2)} \big)
        & \subseteq \Dom(\phi^1).
  \end{align*}
\eit
For $h > 0$ and $x \in \ov{\Dom(\phi^1)} \cap \Dom(\phi^2)$ we define
\begin{align*}
  K_h x := J_h^2 J_h^1 x.
\end{align*}
It follows from $(A_2)$ that $K_h x \in \Dom(\phi).$ In particular, it follows that $K_h$ maps $\Dom(\phi)$ into itself.

A \emph{discretisation} $\h$ is a finite sequence of positive numbers $(h_i)_{i=1}^n \subseteq (0, \infty).$ For $k =1, \ldots, n$ we set
\begin{align*}
   |\h| := \sup_{1\leq k \leq n} h_k, \qquad
    t_\h^0 := 0, \qquad
    t_\h^k := 2 \sum_{j = 1}^k h_j.
\end{align*}
Given $x \in \Dom(\phi)$ and a discretisation $\h = (h_i)_{i=1}^n,$ the discrete scheme is defined for $k =1, \ldots, n$ by
\begin{align*}
x_\h^0 := x, \qquad
    \hx_\h^k := J_{h_k}^1 x_\h^{k-1}, \qquad
      x_\h^k := J_{h_k}^2 \hx_\h^k 
   	    		  = K_{h_k} x_\h^{k-1}.
\end{align*}
We shall associate with the discretisation the piecewise constant function $\ux_{\h} \text{ (resp. $\ox_{\h}$)} : [0,t_\h^n] \to X$ which takes the values $x_\h^k$ at $t_\h^k,$ $0 \leq k \leq n,$ is constant on the intervals $(t_\h^{k-1}, t_\h^k),$ $1 \leq k \leq n,$ and is right-continuous (resp. left-continuous).

To motivate the next assumption, let us remark that, as we have seen above, for $x \in \Dom(\phi^i)$ we have $\phi^i(J_h^i x) - \phi^i(x) \leq 0,$ $i =1,2,$ but in general we do not have any bound for $\phi^1(J_h^2 x) - \phi^1(x)$ with $x \in \Dom(\phi^1).$ The next assumption provides some control on this quantity. For $k = 1, \ldots, n$ we set
\begin{align*}
 \delta_{\h,x}^k := [\phi^1( x_\h^k ) - \phi^1( \hx_\h^k )]^+, \qquad
  \Delta_{\h,x}^k := \sum_{j=1}^k \delta_{\h,x}^k.
\end{align*}
We can now state the assumption:
\bit
\item[($A_3$)]
There exists a lsc functional $\chi : X \to \R \cup \{ +\infty \},$ not identically $+ \infty,$ such that the following holds:
for any $w \in \ov{\Dom(\phi)}$ and any $h_*, T, R, U > 0$ there exists  $K \in (0,\infty)$ such that
for any discretisation $\h = (h_i)_{i=1}^n$ satisfying
\begin{align} \label{eq:quant1}
 |\h|\leq h_*, \qquad
  t_\h^n \leq T,
\end{align}
and any $x \in \Dom(\phi)$ satisfying
\begin{align} \label{eq:quant2}
 d^2(x, w) \leq R, \qquad
   \chi(x) \leq U, \qquad
\end{align}
we have
\begin{align*}
 \Delta_{\h,x}^n \leq K.
\end{align*}
\eit

Now we are ready to state the main result of this paper.

 \begin{theorem}  \label{thm:trotter-formula}
Assume that $(A_1), (A_2), (A_3)$ hold.
Let $T_* > 0$ and let $(\h^i)_{i \geq 1}$ be a sequence of discretisations, where $\h^i = (h_k^i)_{k=1}^{n^i},$ such that
\begin{align*}
(i) \quad \inf_{i\geq 1} t_{\h^i}^{n^i} \geq T_* \qquad \text{ and } \qquad
(ii) \quad
 \lim_{i \to \infty} |\h^i| = 0.
\end{align*}
Let $x \in \Dom(\phi)$ and let $(x^i)_{i\geq 1} \subseteq \Dom(\phi)$ be a sequence satisfying
\begin{align*}
(iii) \quad \lim_{i \to \infty} x^i = x,
   \qquad 
 (iv) \quad  \sup_{i \geq 1} \chi(x^i) < \infty,
\qquad
  (v) \quad   \sup_{i \geq 1} \phi(x^i) < \infty.
\end{align*}
Then the sequences $\{ {\uxi}_{\h^i} \}_{i \geq 1}$ and $\{ {\oxi}_{\h^i} \}_{i \geq 1}$ converge uniformly on $[0,T_*]$ to a continuous function $u : [0, T_*] \to X$ which satisfies $u(0) = x,$  $\phi \circ u \in L^1((0,T_*);\R)$ and
\begin{align} \label{eq:gradientFlow}
 \frac{d}{dt} \frac{1}{2} d^2(u(t), y) \leq
      \phi(y) - \phi(u(t))
\end{align}
in the sense of distributions on $(0,T_*)$ for any $y \in \Dom(\phi).$
 \end{theorem}

\begin{remark}\label{rem:gradientFlow}
More explicitly, \eqref{eq:gradientFlow} means that  \begin{align*}
 - \frac{1}{2}  \int_0^\infty d^2(u(t), y) \zeta'(t) \,dt\leq \int_0^\infty
  \big(\phi(y) - \phi(u(t))\big) \zeta(t) \, dt
 \end{align*}
for any non-negative test function $\zeta \in C_c^\infty((0,\infty);\R).$
Equivalently, see, e.g, \cite{ClD10}, for any $0 < a < b < \infty,$   \begin{align} \label{eq:intSol}
    \frac{1}{2} d^2(u(b), y) -  \frac{1}{2} d^2(u(a), y)
     \leq  (b-a)\phi(y) - \int_a^b  \phi(u(t)) \, dt.
  \end{align}
 \end{remark}

\begin{remark}\label{rem:uniqueness}
Existence and uniqueness of a solution to \eqref{eq:gradientFlow} for $x \in
\ov{\Dom(\phi)}$ has been proved in \cite[Theorem 4.0.4]{AGS}
under suitable coercivity and convexity assumptions which imply $(A_1)$ for $\phi$.
Here we do not assume that $(A_1)$ holds for $\phi.$ Therefore the existence of a gradient flow for $\phi$ does not follow from the results in \cite{AGS}.
\end{remark}

\begin{remark}\label{rem:other-bound} 
As we observed before, in $(A_3)$ we impose a bound for $\phi^1(J_h^2 x) - \phi^1(x)$ with $x \in \Dom(\phi^1).$ Note however that we do \emph{not} assume any bound for $\phi^2(J_h^1 x) - \phi^2(x)$ with $x \in \Dom(\phi^2).$
\end{remark}

\begin{remark}\label{rem:Hilbert}
Convergence of the splitting method is well-known in the case where $X$ is a Hilbert space and each $\phi^i$ is a convex functional \cite{Bre73, KaMa78}. If $X$ is a Hilbert space, then our assumptions are more restrictive than the ones in \cite{KaMa78}.
\end{remark}

\begin{remark}\label{rem:}
It follows from the theory presented in \cite{AGS}, see also \cite{Cl-Bielefeld}, that  for $i = 1, 2,$ $(A_1)$ implies the existence of a semigroup of operators $S_t^i : \ov{\Dom(\phi^i)} \to \Dom(\phi^i),$ $t\geq 0,$ such that for any $x \in \ov{\Dom(\phi^i)}$ the function $u^i(t) := S_t^i x$ satisfies \eqref{eq:gradientFlow} with $\phi = \phi^i.$
It appears from its proof that Theorem \ref{thm:trotter-formula} remains valid if we replace one or both of the resolvents with the associated semigroup.
\end{remark}

At first sight $(A_3)$ may seem difficult to verify in concrete situations. However, the next result provides some sufficient conditions for $(A_3)$ which are easier to state and which will be shown to be fulfilled in a number of examples in Section \ref{sec:applications}.

\begin{proposition}\label{prop:suff-for-A4}
Assume that $(A_1)$ and $(A_2)$ hold and suppose that $\phi^1$ and $\phi^2$ satisfy at least one of the following conditions:
 \ben
  \item
There exists $c \geq 0$ such that for any $h >0$ and $x \in \Dom(\phi)$ we have
  \begin{align*}
 \phi^1(J_h^2 x) &\leq \phi^1(x) + c h;
  \end{align*}
  \item
Assume that $\phi^1[X] \subseteq [0,\infty].$ Moreover, assume that there exists $\alpha \geq 0$ such that for any $h >0$ and $x \in \Dom(\phi)$ we have
  \begin{align*}
 \phi^1(J_h^2 x) &\leq e^{\alpha h}\phi^1(x);
  \end{align*}
 \item
Assume that $\phi^2[X] \subseteq [0,\infty].$ Moreover, assume that there exist $\alpha, c \geq 0$ such that for any $h >0$ and $x \in \Dom(\phi)$ we have
  \begin{align*}
 (i) \qquad \phi^1(J_h^2 x) &\leq \phi^1(x) + c h \phi^2(J_h^2 x),\\
 (ii) \qquad \phi^2(J_h^1 x) &\leq e^{\alpha h}\phi^2(x).
  \end{align*}
 \een
Then $(A_3)$ is satisfied with
\begin{align*}
 \textit{(1)}  \ \chi = constant,  \qquad
 \textit{(2)}  \ \chi = \phi^1,  \qquad
 \textit{(3)} \ \chi = \phi^2.
\end{align*}
\end{proposition}

We apply our results to the case where $X = \calP_2(\R^d),$ the space of probability measures on $\R^d$ with finite second moment endowed with the $L^2$-Wasserstein metric. On this space we consider the sum of the (negative) Boltzmann entropy and a potential energy. The associated gradient flow corresponds to the Fokker-Planck equation \cite{JKO98}. We show that the conditions of Proposition \ref{prop:suff-for-A4}(1) are satisfied under suitable assumptions on the potential, and therefore the splitting method converges in this setting. Using (2) and (3) of Proposition \ref{prop:suff-for-A4}, we obtain similar results by replacing the Boltzmann entropy by the R\'{e}nyi entropy, which corresponds to the porous medium equation \cite{O01}.

The paper is organised as follows. In Section \ref{sec:trotterMetric} we shall work in the abstract setting of a metric space and give the proof of Theorem \ref{thm:trotter-formula} and Proposition \ref{prop:suff-for-A4}. The applications to gradient flows in the Wasserstein space are presented in Section \ref{sec:applications}.

\section{Proof of Theorem \ref{thm:trotter-formula} and of Proposition \ref{prop:suff-for-A4}}
\label{sec:trotterMetric}

We continue working in the setting of Section \ref{sec:intro}. In particular, we assume throughout this section (with the exception of the proof of Proposition \ref{prop:suff-for-A4}) that $(A_1),$ $(A_2)$ and $(A_3)$ hold.

We will adapt the arguments from \cite{AGS} where a single functional has been considered. First we state a simple analogue of \eqref{eq:DEVI}.

 \blem[Discrete Evolution Variational Inequality]
   \label{lem:crucialQuadRef}
Let $x \in \Dom(\phi)$ and let $\h := (h_k)_{k=1}^n$ be a discretisation.
For $w \in \Dom(\phi)$ and $k =1, \ldots, n$ we have
 \beq\bal \label{eq:crucialQuadRef}
\frac{1}{2h_k} \Big( d^2(x_\h^k, w) - d^2(x_\h^{k-1}, w) \Big)
   & \leq \phi(w)  - \phi(x_\h^k) - \frac{1}{4h_k}d^2(x_\h^k, x_\h^{k-1})
  + \delta_{\h,x}^k.
 \eal\eeq
In particular,
\begin{align} \label{eq:special-case}
\frac{3}{4h_k} d^2(x_\h^k, x_\h^{k-1})
   & \leq \phi(x_\h^{k-1}) - \phi(x_\h^k) + \delta_{\h,x}^k.
\end{align}
\elem

 \bpf
Recall that $\hx_\h^k = J_{h_k}^1 x_\h^{k-1}.$ Using \eqref{eq:DEVI} we find that
 \begin{align*} 
  \frac{1}{2h_k} \Big(  d^2(\hx_\h^k, w)
 		- d^2(x_\h^{k-1}, w) \Big)
   &\leq \phi^1(w) - \phi^1(\hx_\h^k)
  	 - \frac{1}{2h_k}d^2(\hx_\h^k, x_\h^{k-1}),
  \\ \frac{1}{2h_k} \Big(  d^2(x_\h^{k}, w)
 - d^2(\hx_\h^k, w) \Big)
   &\leq \phi^2(w) - \phi^2(x_\h^{k})
   	 - \frac{1}{2h_k}d^2(x_\h^{k}, \hx_\h^k).
 \end{align*}
Adding these inequalities, we obtain
\begin{align*}
  \frac{1}{2h_k} \Big(  d^2(x_\h^k, w) &- d^2(x_\h^{k-1}, w) \Big)
   \leq \phi(w) - \phi^2(x_\h^{k}) - \phi^1(x_\h^k)
 	  \\ & \qquad
			+ \phi^1(x_\h^k)	 - \phi^1(\hx_\h^k)
  	 - \frac{1}{2h_k} \Big(d^2(x_\h^{k-1}, w) 
                    + d^2(x_\h^{k}, \hx_\h^k) \Big).
\end{align*}
Finally, observe that
\begin{align*}
 d^2(x_\h^k, x_\h^{k-1})
 	\leq 2  d^2(x_\h^k, \hx_\h^k) + 2  d^2(\hx_\h^k, x_\h^{k-1}).
\end{align*}
 \epf

Our next goal is to prove some a priori estimates in Proposition \ref{prop:a-priori}. We will use the following discrete version of Gronwall's lemma, taken from \cite[Lemma 3.2.4]{AGS}. For the sake of completeness we include the proof.

 \begin{lemma}\label{lem:Gronwall}
Let $A \geq 0,$ and let $\{a_n\}_{n\geq 1}, \{\tau_n\}_{n\geq1}$ be sequences of positive numbers satisfying $m := \sup_{n \geq 1} \tau_n < 1$ and
\begin{align*}
 a_n \leq A + \sum_{k=1}^n \tau_k a_k, \qquad n \geq 1.
\end{align*}
Then, writing $\beta := \frac{1}{1-m},$ $t_0 := 0,$ and $t_n := \sum_{k=1}^n \tau_k$ for $n \geq 1,$ we have
\begin{align} \label{eq:Gronwall}
 a_n \leq A\beta \exp(\beta t_{n-1} ).
\end{align}
 \end{lemma}

 \begin{proof}
We argue by induction and observe that \eqref{eq:Gronwall} clearly holds for $n=1$.

Let $n \geq 1$ and suppose that \eqref{eq:Gronwall} holds for all $1 \leq k \leq n.$ Since, for any $n \geq 1,$
\begin{align*}
 a_n \leq \frac{A}{1-\tau_n} + \frac{1}{1-\tau_n}\sum_{j=1}^{n-1} \tau_j a_j,
\end{align*}
we obtain
\begin{align*}
 a_{n+1}
   & \leq A\beta + \beta\sum_{j=1}^n \tau_j a_j
     \leq A\beta + A\beta^2\sum_{j=1}^n \tau_j e^{\beta t_{j-1}}
 \\& \leq A\beta + A\beta^2\sum_{j=1}^n
        \int_{t_{j-1}}^{t_j} e^{\beta t} \,dt
       = A\beta + A\beta^2 \int_0^{t_n} e^{\beta t} \,dt
 \\&    = A\beta e^{\beta t_n},
\end{align*}
which completes the induction step.
 \end{proof}

 \begin{proposition}[A priori estimates] \label{prop:a-priori}
Let  $w \in \ov{\Dom(\phi)}$ and $\thh, K, R, S, T > 0$ be given. There exist constants $C, \widetilde C  \in (0,\infty)$ such that for every $x \in \Dom(\phi)$ and any discretisation $\h := (h_k)_{k=1}^n$ satisfying
\begin{align}\label{eq:RST}
 |\h|\leq \frac18 \thh, \qquad
       \Delta_{\h,x}^n \leq K, \qquad
 d^2(x, w) \leq R, \qquad
   \phi(x) \leq S, \qquad
 t_\h^n \leq T,
\end{align}
we have
\begin{align} \label{eq:a-priori-one}
 d^2( x_\h^n, w) &\leq C,\\  \label{eq:a-priori-two}
 \frac34 \sum_{k=1}^n \frac{1}{h_k} d^2( x_\h^k, x_\h^{k-1})
  \leq \phi(x_\h^0) - \phi(x_\h^n) + \Delta_{\h,x}^n
  & \leq \widetilde C.
\end{align}
 \end{proposition}
 
Notice that, by applying the first inequality of \eqref{eq:a-priori-two} to the subdiscretisation $(h_j)_{j=1}^k$, $k = 1, \ldots, n,$ we have
\begin{align} \label{eq:deelpartitie}
 \phi(x_\h^k) 
  \leq \phi(x_\h^0) + \Delta_{\h,x}^k
  \leq \phi(x_\h^0) + \Delta_{\h,x}^n
\end{align}
for any $k =1, \ldots, n.$

 \begin{proof}
It follows from \eqref{eq:special-case} that, for $k =1, \ldots, n,$
\begin{align*}
 \frac{3}{4 h_k} d^2(x_\h^k, x_\h^{k-1})
    \leq \phi(x_\h^{k-1}) - \phi(x_\h^k) + \delta_{\h,x}^k.
\end{align*}
Summation over $k$ yields the first inequality in \eqref{eq:a-priori-two}.

For $\eps > 0$ we obtain, setting $h_0 = h_{n+1} = 0,$
\begin{align*}
   d^2&(x_\h^n, w) - d^2(x_\h^0, w)
\\& = \sum_{k=1}^n  d^2(x_\h^k, w) - d^2(x_\h^{k-1}, w)
\\& = \sum_{k=1}^n \big(d(x_\h^k, w) - d(x_\h^{k-1}, w) \big)
                            \big(d(x_\h^k, w) + d(x_\h^{k-1}, w) \big)
\\& \leq \sum_{k=1}^n d(x_\h^k, x_\h^{k-1})
                            \big(d(x_\h^k, w) + d(x_\h^{k-1}, w) \big)
\\& \leq \sum_{k=1}^n \frac{\eps}{h_k}d^2(x_\h^k, x_\h^{k-1})
       + \sum_{k=1}^n \frac{h_k}{4\eps}
                            \big(d(x_\h^k, w) + d(x_\h^{k-1}, w) \big)^2
\\& \leq \sum_{k=1}^n \frac{\eps}{h_k}d^2(x_\h^k, x_\h^{k-1})
       + \sum_{k=1}^n \frac{h_k}{2\eps}
                            \big(d^2(x_\h^k, w) + d^2(x_\h^{k-1}, w) \big)
\\&    = \sum_{k=1}^n \frac{\eps}{h_k}d^2(x_\h^k, x_\h^{k-1})
       + \sum_{k=0}^n\frac{ h_k + h_{k+1}}{2\eps} d^2(x_\h^k, w).
\end{align*}
Combining this estimate with the first inequality in \eqref{eq:a-priori-two} (which we already proved), we arrive at
\beq\bal \label{eq:new}
 \frac34 ( d^2(x_\h^n, w) - d^2(x_\h^0, w) )
  & \leq \eps ( \phi(x_\h^0) - \phi(x_\h^n) + \Delta_{\h,x}^n )
  \\&\quad   + \frac34 \sum_{k=0}^n \frac{ h_k + h_{k+1} }{2\eps} d^2(x_\h^k, w).
\eal\eeq
For $i =1,2,$ $h > 0,$ and $z \in X$ we set
\begin{align*}
 \Phi^i(h,z;y) :=
    \left\{ \begin{array}{ll}
 \phi^i(y) + \frac{1}{2h} d^2(z,y),
 & \text{$y \in \Dom(\phi^i)$},\\
 + \infty,
 & \text{otherwise,}
 \end{array} \right.
\end{align*}
and
\begin{align*}
  \hat \phi_h(w) :=
     \Phi^1(h,w;J_h^1 w) + \Phi^2(h,w;J_h^2 w).
\end{align*}
The defining property of $J_h^i w$ implies that
\begin{align*}
 \Phi^i(h,w; J_h^i w) \leq \Phi^i(h,w; x_\h^n).
\end{align*}
Adding these inequalities for  $i = 1,2,$ it follows that
\begin{align} \label{eq:new2}
\hat\phi_{h}(w) \leq \phi(x_\h^n) + \frac{1}{h}d^2(x_\h^n,
w).
\end{align}
Substituting  $\eps := \frac{\thh}{2}$ in \eqref{eq:new}, and using
\eqref{eq:RST} and \eqref{eq:new2}, we obtain
\begin{align*}
  \frac34 \big( d^2(x_\h^n, w) - d^2(x_\h^0, w) \big)
  &\leq  \frac{\thh}{2}\big( S - \hat\phi_{\thh}(w)
       +  \frac{1}{\thh}d^2(x_\h^n, w)
       + \Delta_{\h,x}^n \big)
   \\&  \quad  + \frac34 \sum_{k=0}^n
        \frac{ h_k + h_{k+1} }{\thh} d^2(x_\h^k, w).
\end{align*}
Rearranging terms, using \eqref{eq:RST}, and multiplying the
inequality by $4,$ yields
\begin{align*}
   d^2(x_\h^n, w)
     &\leq A
       +  \sum_{k=0}^n \tau_k d^2(x_\h^k, w),
\end{align*}
where $A := [3 R + 2 \thh ( S - \hat\phi_{\thh}(w) + K)]^+$ and $\tau_k := 3 \frac{ h_k + h_{k+1} }{\thh}
\leq \frac34.$ Applying Lemma \ref{lem:Gronwall} we obtain
\begin{align*}
 d^2(x_\h^n, w)
  &\leq 4A \exp\Big( 12 \sum_{k=1}^{n-1} \frac{h_k + h_{k+1}}{\thh}\Big)
  \leq 4A \exp\Big(\frac{24}{\thh}\sum_{k=1}^n h_k\Big)
  \\&\leq 4A \exp \Big( \frac{12T}{\thh} \Big)
     =: C,
\end{align*}
which proves \eqref{eq:a-priori-one}.

Finally, using \eqref{eq:RST}, \eqref{eq:a-priori-one}, and \eqref{eq:new2}, we
obtain
\begin{align*}
 \phi(x_\h^0) - \phi(x_\h^n)  + \Delta_{\h,x}^n
 & \leq S - \hat\phi_{\thh}(w) + \frac{1}{\thh}d^2(x_\h^n, w)
         + K
 \\& \leq S - \hat\phi_{\thh}(w) + \frac{C}{\thh} + K
      =: \widetilde C.
\end{align*}
which proves the second inequality in
\eqref{eq:a-priori-two}.
 \end{proof}

Let $x \in \Dom(\phi)$ and a discretisation $\h :=  (h_k)_{k=1}^n \subseteq (0,\infty)$  be given. As in \cite{AGS} it will be useful to consider continuous interpolants of some relevant quantities which are originally defined only on the discrete set $(t_{\h}^j)_{j=0}^n.$ For this purpose we will use the (unique) function $\ell_\h : [0,t_\h^n] \to [0,1]$ which is affine on each interval $[t_{j-1}, t_j),$ $j = 1, \ldots, n,$ and
satisfies $\ell_\h(t_\h^j) = 0$ for $j = 0, \ldots, n,$ and $\lim_{t \uparrow t_\h^j} \ell_\h(t) =1$ for $j = 1, \ldots, n.$
We consider the function $d_{\h,x} : [0,t_\h^n] \times X \to \R$
defined by
\begin{align*}
 d_{\h,x}^2(t,y) :=
 \big(1- \ell_\h(t)\big) d^2(\ux_\h(t), y)
                 + \ell_\h(t) d^2(\ox_\h(t), y)
\end{align*}
and the function $\phi_{\h,x} : [0,t_\h^n] \to \R$
defined by
\begin{align*}
 \phi_{\h,x}(t) :=
             \big( 1- \ell_\h(t) \big) \phi( \ux_\h(t) )
                 + \ell_\h(t) \phi( \ox_\h(t) ).
\end{align*}
Note that $\ux_\h(t), \ox_\h(t) \in \Dom(\phi),$ since $K_h$ maps $\Dom(\phi)$ into itself, as has already been observed before.
Finally, we define the function $R_{\h,x} : [0,t_\h^n] \to \R$ by
\begin{align*}
 R_{\h,x}(t) &:= \sum_{k=1}^n \one_{[t_\h^{k-1}, t_\h^k)}(t)
              \bigg((1- \ell_\h(t))
                 \Big(\phi(x_\h^{k-1})
                 - \phi(x_\h^k) + \delta_{\h,x}^k
  \\ &  \qquad        - \frac{1}{4h_k}d^2(x_\h^k, x_\h^{k-1})\Big)
      + \ell_\h(t)
              \Big( \delta_{\h,x}^k - \frac{1}{4h_k}d^2(x_\h^k, x_\h^{k-1})\Big)\bigg).
\end{align*}

The following result is an analogue of \cite[Theorem 4.1.4]{AGS}.

 \begin{lemma}[Gradient flow approximation] \label{lem:Rbound}
Let $x \in \Dom(\phi)$ and a discretisation $\h :=  (h_k)_{k=1}^n$ be given. For every $y \in \Dom(\phi)$ and every $t \in [0,t_\h^n] \setminus \{ t_\h^0, \ldots, t_\h^n\}$ we have
\begin{align} \label{eq:Rbound}
 \frac{d}{d t} d_{\h,x}^2(t,y) + \phi_{\h,x}(t) - \phi(y)
   \leq R_{\h,x}(t),
\end{align}
where $\frac{d}{dt}$ denotes the pointwise derivative.
 \end{lemma}

 \begin{proof}
First we remark that $t \mapsto d_{\h,x}^2(t, y)$ is a piecewise affine function. As a consequence, the derivative in \eqref{eq:Rbound} exists for $t$ in the above-mentioned set.
For $k = 1, \ldots, n$ and $t \in (t_\h^{k-1}, t_\h^k)$ we obtain using \eqref{eq:crucialQuadRef},
 \begin{align*}
 \frac{d}{dt} &d_{\h,x}^2(t,y) + \phi_{\h,x}(t) - \phi(y)
  \\&= \frac{1}{2 h_k} (d^2(x_\h^k,y) - d^2(x_\h^{k-1},y))
       + \phi_{\h,x}(t) - \phi(y)
  \\&\leq \delta_{\h,x}^k - \frac{1}{4 h_k} d^2(x_\h^k, x_\h^{k-1})
       + \phi_{\h,x}(t) - \phi(x_\h^k)
  \\&   = \delta_{\h,x}^k - \frac{1}{4 h_k} d^2(x_\h^k, x_\h^{k-1})
       + (1- \ell_\h(t)) (\phi(x_\h^{k-1})- \phi(x_\h^k))
  \\& = R_{\h,x}(t).
\end{align*}
 \end{proof}

The following estimate will be useful in the proof of Proposition \ref{prop:trotter-est} below.

 \begin{lemma} \label{lem:R-estimate}
Let $x\in \Dom(\phi)$ and a discretisation $\h := (h_k)_{k=1}^n$ be given. For $1 \leq k \leq n,$
\begin{align*}
 \int_0^{t_\h^k} [R_{\h,x}(s)]^+ \, ds
     \leq  |\h|\big(\phi(x_\h^0) - \phi(x_\h^k)
          + 2 \Delta_{\h,x}^k \big).
\end{align*}
 \end{lemma}

 \begin{proof}
For $i = 1, \ldots, k$, we have as a consequence of \eqref{eq:special-case},
\begin{align} \label{eq:nonneg}
  \phi(x_\h^{i-1}) - \phi(x_\h^i) + \delta_{\h,x}^i
      - \frac{1}{4 h_i} d^2(x_\h^i , x_\h^{i-1})
        \geq \frac{1}{2 h_i} d^2(x_\h^i, x_\h^{i-1}) \geq 0,
\end{align}
and therefore, for $s \in [t_\h^{i-1}, t_\h^i)$ we obtain
\beq\bal \label{eq:Rplus}
  [R_{\h,x}(s)]^+ & \leq (1- \ell_\h(t))
    \Big(   \phi(x_\h^{i-1}) - \phi(x_\h^i) + \delta_{\h,x}^i
     \Big)
  + \ell_\h(t)  \delta_{\h,x}^i.
\eal\eeq
Observe that $\delta_{\h,x}^k$ and  $\phi(x_\h^{i-1}) - \phi(x_\h^i) +  \delta_{\h,x}^i$ are non-negative, as follows from \eqref{eq:nonneg}.
Combining this with \eqref{eq:Rplus} and the identities 
\begin{align*}
 \int_{t_\h^{i-1}}^{t_\h^i} \ell_\h(t) \, dt =  \int_{t_\h^{i-1}}^{t_\h^i} 1- \ell_\h(t) \, dt = h_i,
\end{align*}
we obtain
\begin{align*}
 \int_0^{t_\h^k} [R_{\h,x}(s)]^+ \, ds
  & = \sum_{i=1}^k
  \int_{t_\h^{i-1}}^{t_\h^i}  [R_{\h,x}(s)]^+ \, ds
 \\&  \leq \sum_{i=1}^k
   h_i \big( \phi(x_\h^{i-1}) - \phi(x_\h^i) +  \delta_{\h,x}^i \big)
 + h_i \delta_{\h,x}^i
 \\&  \leq  |\h| \sum_{i=1}^k
   \big( \phi(x_\h^{i-1}) - \phi(x_\h^i) +  \delta_{\h,x}^i \big)
 + |\h| \delta_{\h,x}^i
 \\& \leq |\h|\big(\phi(x_\h^0) - \phi(x_\h^k)
          + 2 \Delta_{\h,x}^k \big).
\end{align*}

 \end{proof}

We will now compare the discrete scheme induced by $(\h,x)$ to another discrete scheme induced by $(\r,y),$ where $y \in \Dom(\phi)$ and $\r = (r_j)_{j=1}^m$ is a discretisation.
For this purpose we consider the continuous function $d_{\h\r}^2: [0,t_\h^n] \times [0,t_\r^m] \to \R$ defined by
\begin{align*}
 d_{\h\r}^2(t,s) := (1-\ell_\r(s)) d_{\h,x}^2(t,\uy_\r(s))
                    + \ell_\r(s) d_{\h,x}^2(t, \oy_\r(s)).
\end{align*}
In this formula the dependence of $ d_{\h\r}^2(t,s)$ on $x$ and $y$ is suppressed in the notation. With this notation we have the following result:

 \begin{corollary} \label{cor:d-est}
For all $t \in [0,\min\{ t_\h^n, t_\r^m \}]$ we have
\begin{align*}
d_{\h\r}^2(t,t) \leq d^2(x_\h^0,y_\r^0)
  + \int_0^t R_{\h,x}(s) + R_{\r,y}(s) \, ds.
\end{align*}
 \end{corollary}

 \begin{proof}
For each fixed $s \in [0,t_\r^m]$ we obtain for all $t \in [0,t_\h^n] \setminus\{t_\h^0, \ldots, t_\h^n\}$ by \eqref{eq:Rbound},
\begin{align*}
 \frac{\partial}{\partial t} d_{\h\r}^2(t,s) + \phi_{\h,x}(t) - \phi_{\r,y}(s) \leq R_{\h,x}(t).
\end{align*}
Similarly, reversing the roles of $(\h,x)$ and $(\r,y),$ yields for  fixed $t \in [0,t_\r^m]$ and for all $s \in [0,t_\r^m] \setminus\{t_\r^0, \ldots, t_\r^m\},$
\begin{align*}
 \frac{\partial}{\partial s} d_{\r\h}^2(s,t) + \phi_{\r,y}(s) - \phi_{\h,x}(t)
    \leq R_{\r,y}(s).
\end{align*}
Noting that $ d_{\h\r}^2(t,s) = d_{\r\h}^2(s,t),$ we obtain by adding these inequalities, for each $t \in [0,\min\{ t_\h^n, t_\r^m\}] \setminus  \{t_\h^0, \ldots, t_\h^n,  t_\r^0, \ldots, t_\r^m\},$
 \begin{align*}
\frac{d}{d t} d_{\h\r}^2(t,t)
    \leq R_{\h,x}(t) + R_{\r,y}(t).
\end{align*}
Taking into account that $t \mapsto  d_{\h\r}^2(t,t)$ is continuous and piecewise $C^1,$ the result follows by integrating this inequality.
 \end{proof}

The next result contains the main estimate for the proof of the Trotter product formula in Theorem \ref{thm:trotter-formula} below.

 \bpr \label{prop:trotter-est}
Let $w \in \ov{\Dom(\phi)}$ and $\thh, R,S,U,T > 0$ be given. There exists a constant $\widetilde K \in (0,\infty)$ such that for all $x, y \in \Dom(\phi)$ and all discretisations $\h := (h_k)_{k=1}^n$ and  $\r := (r_j)_{j=1}^m $
satisfying
\begin{align*}
 |\h|\leq \frac18 \thh, \qquad
 d^2(x_\h^0, w) \leq R, \qquad
 \phi(x_\h^0) \leq S, \qquad
       \chi(x_\h^0) \leq U, \qquad
 t_\h^n \leq T;\\
  |\r|\leq \frac18 \thh, \qquad
 d^2(y_\r^0, w) \leq R, \qquad
 \phi(y_\r^0) \leq S, \qquad
 \chi(y_\r^0) \leq U, \qquad
 t_\r^m \leq T,
\end{align*}
we have, for $t \in [0, \min\{t_\h^n, t_\r^m\}],$
\begin{align*}
  d^2(\ox_\h(t), \oy_\r(t))
   \leq  \widetilde K \big(d^2(x, y) + |\h| + |\r| \big).
\end{align*}
 \epr

 \bpf
Let $t \in [0, \min\{t_\h^n, t_\r^m\}]$ and let $k,j \geq 1$ be such that $t \in [t_\h^{k-1},t_\h^k) \cap
[t_\r^{j-1}, t_\r^j),$ where we use the convention that $t_\h^{n+1} := t_\h^n+ 1$ and $t_\r^{m+1} := t_\r^m +1.$ To simplify notation we write
\begin{align*}
  a_{\ot \ut} := d^2(x_\h^k, x_\h^{k-1}),  \qquad
     a_{\os \us} := d^2(y_\r^j, y_\r^{j-1}), \\
  a_{\ut \us} := d^2(x_\h^{k-1}, y_\r^{j-1}), \qquad
     a_{\ot \us} := d^2(x_\h^k, y_\r^{j-1}), \\
  a_{\ut \os} := d^2(x_\h^{k-1}, y_\r^j),  \qquad
     a_{\ot \os} := d^2(x_\h^k, y_\r^j).
\end{align*}
With this notation we have (using $(\sum_{i=1}^n b_i)^2 \leq n \sum_{i=1}^n  b_i^2$),
\begin{align*}
   & d^2(\ox_\h(t), \oy_\r(t))
    = d^2(x_\h^k, y_\r^j)
 \\&    = \big((1-\ell_\h(t))(1-\ell_\r(t))
        + (1- \ell_\h(t))\ell_\r(t)
 \\& \quad
        + \ell_\h(t)(1-\ell_\r(t)) + \ell_\h(t) \ell_\r(t) \big) a_{\ot \os}
 \\& \leq 3(1-\ell_\h(t))(1-\ell_\r(t))
               ( a_{\ot \ut} + a_{\ut \us}  +   a_{\os \us} )
 \\& \quad
        + 2 (1- \ell_\h(t))\ell_\r(t) (a_{\ot \ut} + a_{\ut \os})
 \\& \quad
        + 2 \ell_\h(t)(1-\ell_\r(t)) (a_{\ot \us} + a_{\os \us})
        + \ell_\h(t) \ell_\r(t) a_{\ot \os}
 \\& \leq 3(1-\ell_\h(t))(1-\ell_\r(t))
           ( a_{\ot \ut} + a_{\ut \us} + a_{\os \us} )
 \\& \quad
        + 3 (1- \ell_\h(t))\ell_\r(t) (a_{\ot \ut} + a_{\ut \os})
 \\& \quad
        + 3 \ell_\h(t)(1-\ell_\r(t)) (a_{\ot \us} + a_{\os \us})
        + 3 \ell_\h(t) \ell_\r(t) a_{\ot \os}
 \\&    = 3 d_{\h\r}^2(t,t)
        + 3(1-\ell_\h(t)) d^2(x_\h^k, x_\h^{k-1})
        + 3(1-\ell_\r(t)) d^2(y_\r^j, y_\r^{j-1}).
\end{align*}
To complete the proof, we will estimate each of the terms at the right-hand
side. Using Corollary \ref{cor:d-est}, Lemma \ref{lem:R-estimate}, Proposition \ref{prop:a-priori} and $(A_3)$ with $h_* = \frac18 \thh,$
 \begin{align*}
   &    d_{\h\r}^2(t,t)
 \\& \leq d^2(x_\h^0, y_\r^0)
        + \int_0^t R_{\h,x}(s) + R_{\r,y}(s)  \, ds
 \\& \leq d^2(x_\h^0, y_\r^0)
        + \int_0^{t_\h^k} [R_{\h,x}(s)]^+  \, ds
        + \int_0^{t_\r^j} [R_{\r,y}(s)]^+\, ds
 \\& \leq d^2(x_\h^0, y_\r^0)
        + |\h|\big(\phi(x_\h^0) - \phi(x_\h^k)
          + 2 \Delta_{\h,x}^k\big)
        +|\r|\big(\phi(y_\r^0) - \phi(y_\r^j)
          + 2 \Delta_{\r,y}^j \big)
 \\& \leq d^2(x_\h^0, y_\r^0)
        + (\widetilde C + K ) ( |\h| + |\r| ),
 \end{align*}
where $\tilde C$ and $K$ are the constants from \eqref{eq:a-priori-two} and $(A_3)$ respectively.
By another application of \eqref{eq:a-priori-two},
 \begin{align*}
  d^2(x_\h^k, x_\h^{k-1})
   \leq |\h| \sum_{j=1}^n \frac{1}{h_j} d^2(x_\h^j, x_\h^{j-1})
   \leq \frac 43 \widetilde C |\h|,
 \end{align*}
and similarly,
 \begin{align*}
  d^2(y_\r^j, y_\r^{j-1})
   \leq \frac 43 \widetilde C |\r|,
 \end{align*}
which completes the proof.
 \epf

The following elementary lemma will be used in the proof of Theorem \ref{thm:trotter-formula} below.

\begin{lemma}\label{lem:sequences}
Let $\{a_i\}_{i \geq 1}$ and $\{b_i\}_{i \geq 1}$ be sequences in $X$ converging to the same limit $c \in X.$ Let $\{\lambda_i\}_{i \geq 1}$ be a sequence in $[0,1],$ and let $\psi : X \to \R \cup \{ + \infty \}$ be lsc functional which is not identically $+\infty.$ Then 
\begin{align*}
 \psi(c) \leq \liminf_{i \to \infty} \big((1 - \lambda_i) \psi(a_i) + \lambda_i \psi(b_i)\big).
\end{align*}
\end{lemma}

\begin{proof}

Suppose that  $\psi(c) \in \R$ (resp. $\psi(c) = + \infty.$)
Let $\eps  > 0$ (resp. let $M > 0$).
Since $\psi$ is lsc, we can find $\delta > 0$ such that $\psi(x) \geq \psi(c) - \eps$ (resp. $\psi(x) \geq M$) whenever $d(c,x) < \delta.$ 
Since $a_i, b_i \to x$ as $i \to \infty,$ we can take $N \geq 1$ such that $d(a_i,x) < \delta$ and $d(b_i, x) < \delta$ for all $i \geq N.$
Consequently, for all $i \geq N$ we have $\psi(a_i) \geq \psi(c) - \eps$ (resp. $\psi(a_i) \geq M$) and $\psi(b_i) \geq \psi(c) - \eps$ (resp. $\psi(b_i) \geq M$).
It follows that $(1 - \lambda_i) \psi(a_i) + \lambda_i \psi(b_i) \geq \psi(c) - \eps$ (resp. $(1 - \lambda_i) \psi(a_i) + \lambda_i \psi(b_i) \geq M$), which implies the result.
\end{proof}

 \bpf[Proof of Theorem \ref{thm:trotter-formula}]
Note that the sequences  $\{ \uxi_{\h^i}(t) \}_{i \geq 1}$ and $\{ \oxi_{\h^i}(t) \}_{i \geq 1}$ are well-defined on $[0,T_*]$ as a consequence of $(i).$
Proposition \ref{prop:trotter-est} and $(ii)$ imply that the sequence $\{ \oxi_{\h^i}(t) \}_{i \geq 1}$ is a Cauchy sequence, even uniformly in $t \in [0,T_*].$
Using the completeness of $(X, d)$ there exists a limit $u(t),$ $t \in [0,T_*],$ which is right-continuous. From \eqref{eq:a-priori-two} we have $d^2( x_\h^k, x_\h^{k-1}) \leq \frac43 \widetilde{C} h_k.$ This implies that the sequence $\{ \uxi_{\h^i}(t) \}_{i \geq 1}$ converges to the same limit $u(t),$ $t \geq 0,$ which is also left-continuous on $[0,T_*].$ Assumption $(iii)$ implies that $u(0) =x.$

Next we show that $\phi \circ u \in L^1((0,T_*); \R).$ Since $\phi$ is lsc and $u$ is continuous, the function $\phi \circ u : [0,T_*] \to \R$ is lsc, hence Borel measurable and bounded from below.
From \eqref{eq:deelpartitie} we obtain for $t \in [0,T_*],$ 
\begin{align*}
  \phi( \oxi_{\h^i}(t) ) \leq \phi(x^i) + \Delta_{\h^i, x^i}^{n^i}.
\end{align*}
Assumption $(v)$ implies $K_1 := \sup_{i \geq 1} \phi(x^i) < \infty.$ Using assumption $(A_3)$ with $h_* := \sup_{i \geq 1} |\h^i|,$ $T := T_* + h_*,$ $w := x,$ $R := \sup_{i \geq 1} d^2(x^i, x),$ $U := \sup_{i \geq 1} \chi(x^i),$ which is finite by $(iv),$ we have $K_2 := \sup_{i \geq 1} \Delta_{\h^i, x^i}^{n^i} < \infty.$
Therefore
\begin{align*}
 \phi(u(t)) \leq \liminf_{i \to \infty} \phi( \oxi_{\h^i}(t) )
   \leq \sup_{i \geq 1} \phi(x^i) 
       +\sup_{i \geq 1}  \Delta_{\h^i, x^i}^{n^i} 
       \leq K_1 + K_2. 
\end{align*}
It follows that $\phi \circ u$ is bounded from above, which implies the claim.

It remains to show that $(u(t))_{t > 0}$ is a solution to \eqref{eq:gradientFlow}. This will be done by passing to the limit in \eqref{eq:Rbound}. 
Let $\zeta \in C_c^\infty((0,T_*);\R)$ be non-negative. Take $y \in \Dom(\phi)$ and note that, by what we just proved,
\begin{align*}
 \lim_{i \to \infty} d_{\h^i,x^i}^2(t,y) = d^2(u(t), y), \quad
   \text{ uniformly on $[0,T_*].$}
\end{align*}
Consequently, the mapping $t \mapsto \zeta'(t) d^2(u(t), y)$ is continuous (hence integrable) on $[0,T_*]$ and
\begin{align} \label{eq:final-one}
     \lim_{i \to \infty}  \int_0^{T_*}  \zeta'(t) d_{\h^i,x^i}^2(t,y) \,dt
  & =  \int_0^{T_*}  \zeta'(t) d^2(u(t), y) \, dt.
\end{align}
Using the second inequality of \eqref{eq:a-priori-two} and the lower semicontinuity of $\phi,$ we infer that there exist constants $C, \widetilde C \in \R$ not depending on $t \in [0,T_*]$ and $i \geq 1$ such that
\begin{align*}
 \phi_{\h^i,x^i}(t) \geq \phi(x^i) - \widetilde C \geq C.
\end{align*}
Since $\zeta$ is non-negative, it thus follows that the functions $t \mapsto \zeta(t) \phi_{\h^i,x^i}(t)$ are bounded from below, uniformly in $i.$ Therefore we may apply Fatou's Lemma to obtain
\begin{align*}
  \liminf_{i \to \infty}  \int_0^{T_*}
                  \zeta(t) \phi_{\h^i,x^i}(t) \, dt
    \geq \int_0^{T_*} \liminf_{i \to \infty}
                  \zeta(t) \phi_{\h^i,x^i}(t) \, dt.
\end{align*} 
Applying Lemma \ref{lem:sequences} with $a_i := \uxi_{\h^i}(t) ,$ $b_i :=  \oxi_{\h^i}(t) ,$ $\lambda_i := \ell_{\h^i}(t),$ and $\psi := \phi,$ we infer that 
\beq \bal \label{eq:final-two}
  \liminf_{i \to \infty}  \int_0^{T_*}
                  \zeta(t) \phi_{\h^i,x^i}(t) \, dt
 \geq \int_0^{T_*} \zeta(t) \phi(u(t)) \, dt.
\eal \eeq
Combining \eqref{eq:final-one} and \eqref{eq:final-two}, integrating by parts, using Lemma \ref{lem:Rbound}, Lemma \ref{lem:R-estimate}, \eqref{eq:a-priori-two} and $(A_3),$  we arrive at
\begin{align*}
  \int_0^{T_*} &
         -  \zeta'(t) d^2(u(t), y)
          +   \zeta(t)  \phi(u(t)) \, dt
 \\ & \leq
  \liminf_{i \to \infty}  \int_0^{T_*}
         -  \zeta'(t) d_{\h^i,x^i}^2(t,y)
          +   \zeta(t) \phi_{\h^i,x^i}(t) \, dt
 \\ & =
  \liminf_{i \to \infty}  \int_0^{T_*}
         \zeta(t) \bigg( \frac{d}{dt}d_{\h^i,x^i}^2(t,y)
          +   \phi_{\h^i,x^i}(t) \bigg) \, dt
 \\ & =
  \liminf_{i \to \infty}  \int_0^{T_*}
         \zeta(t) \big( \phi(y) +  R_{\h^i,x^i}(t) \big) \, dt
 \\ & \leq \liminf_{i \to \infty}  \int_0^{T_*}
         \zeta(t) \big(\phi(y)
          + (\widetilde C + K)  |\h^i| \big)
              \, dt
 \\ & =  \int_0^{T_*}
         \zeta(t) \phi(y) \, dt,
\end{align*}
which shows in view of Remark \ref{rem:gradientFlow} that $(u(t))_{t > 0}$ satisfies \eqref{eq:gradientFlow}.
 \epf

 \bpf[Proof of Proposition \ref{prop:suff-for-A4}]
Let $w \in \ov{\Dom(\phi)}$ and  $h_*, T, R, U  > 0$ be given, and take  $x \in \Dom(\phi)$ and a discretisation $\h \subseteq (0, \infty)$ satisfying \eqref{eq:quant1} and \eqref{eq:quant2}.

(1)
Since $c, h_k \geq 0$ for $k =1, \ldots, n,$ we have
\begin{align*}
 \Delta_{\h, x}^n
   &  = \sum_{k=1}^n [\phi^1(x_\h^k)
                        - \phi^1(\hx_\h^k)]^+
\leq c \sum_{k=1}^n h_k
 \leq \frac12cT,
\end{align*}
which implies $(A_3)$ with $\chi = constant.$

(2)
By assumption we have $\phi^1(x_\h^k) \leq e^{\alpha h_k} \phi^1(\hx_\h^k)$ for $k =1, \ldots, n,$ and therefore
\begin{align*}
 \phi^1(x_\h^k) - \phi^1(\hx_\h^k)
  \leq (e^{\alpha h_k} -1) \phi^1(\hx_\h^k).
\end{align*}
Since the right-hand side is non-negative, we have 
\begin{align*}
 [\phi^1(x_\h^k) - \phi^1(\hx_\h^k)]^+
  \leq (e^{\alpha h_k} -1) \phi^1(\hx_\h^k).
\end{align*}
Moreover, it follows by induction that $\phi^1(\hx_\h^k) \leq e^{\frac12 \alpha t_\h^{k-1}} \phi^1(x)$ for $k =1, \ldots, n.$ Consequently,
\begin{align*}
 \Delta_{\h, x}^n
   &  = \sum_{k=1}^n [\phi^1(x_\h^k)
                        - \phi^1(\hx_\h^k)]^+
   \leq \sum_{k=1}^n (e^{\alpha h_k} - 1) \phi^1(\hx_\h^k)
  \\&
   \leq \sum_{k=1}^n (e^{\alpha h_k} - 1)
                        e^{\frac12 \alpha t_\h^{k-1}} \phi^1(x)
      = \sum_{k=1}^n (e^{\frac12 \alpha t_\h^k}
                          - e^{\frac12 \alpha t_\h^{k-1}})
                        \phi^1(x)
  \\& \leq (e^{\frac12 \alpha T} -1 )\phi^1(x),
\end{align*}
which proves $(A_3)$ with $\chi = \phi^1.$

(3)
Note that $(i)$ implies that $\delta_{\h,x}^k =  [\phi^1(x_\h^k)
                        - \phi^1(\hx_\h^k)]^+ \leq c h_k \phi^2(x_\h^k),$ since $c h_k \phi^2(x_\h^k)$ is non-negative. Moreover, using $(ii)$ and induction it follows that $\phi^2(x_\h^k) \leq e^{\frac12 \alpha t_\h^k} \phi^2(x).$ Using these estimates we obtain
\begin{align*}
 \Delta_{\h, x}^n
   & \leq c \sum_{k=1}^n h_k \phi^2(x_\h^k)
     \leq c \phi^2(x) \sum_{k=1}^n h_k e^{\frac12 \alpha t_\h^k}
 \\& \leq c \phi^2(x) \int_{h_1}^{\frac12 t_\h^n + h_n} e^{\alpha s} \, ds
     \leq c \phi^2(x) \int_0^T e^{\alpha s} \, ds,
\end{align*}
which proves $(A_3)$ with $\chi = \phi^2.$
 \epf

 \section{Applications} \label{sec:applications}

In this section we will apply the Trotter product formula from Theorem \ref{thm:trotter-formula} in several concrete situations.

 \subsection*{The Wasserstein space}

Let $\calP_2(\R^d)$ denote the set of all Borel
probability measures $\mu$ on $\R^d,$ $d \geq 1,$ satisfying $\int_{\R^d}
|x|^2\, d\mu(x) < \infty.$ We consider the $L^2$-Wasserstein distance defined for $\mu, \nu \in \calP_2(\R^d)$  by
 \begin{align} \label{eq:WassersteinDistance}
  W_2(\mu,\nu) :=
  \inf\bigg\{ \bigg(\int_{\R^{d} \times \R^d}
   |x_1-x_2|^2 \, d \Sigma(x_1, x_2) \bigg)^{1/2} :
     \Sigma \in \G(\mu,\nu)  \bigg\}.
 \end{align}
Here $\G(\mu,\nu)$ denotes the collection of probability measures
$\S$ on $\R^{d} \times \R^d$ with marginals $\mu$ and $\nu,$ i.e.,
for all Borel sets $A, B \subseteq \R^d,$
 \begin{align*}
   \S(A \times \R^d) = \mu(A), \qquad    \S(\R^d \times B) = \nu(B).
 \end{align*}
Endowed with the metric $W_2,$ $\calP_2(\R^d)$ is a complete
separable metric space. For a Borel mapping $T : \R^d \to
\R^n$ and a Borel probability measure $\mu$ on $\R^d$ we write
$T_\# \mu$ to denote the image measure on $\R^n$ defined by
$T_\# \mu (B) := \mu(T^{-1}(B))$ for a Borel set $B \subseteq
\R^n.$

The infimum in \eqref{eq:WassersteinDistance} is attained
(see, e.g., \cite[Theorem 4.1]{Vil08}). Moreover, a celebrated
result by Brenier, independently due to Rachev and
R\"uschendorf and later refined by McCann, asserts that if
$\mu$ is absolutely continuous with respect to the Lebesgue
measure $\mathscr{L}^d,$ the minimizer $\S \in \G(\mu,\nu)$ is
unique and can be written as $\S = (I \times \nabla f)_\# \mu$
for some convex function $f: \R^d \to \R.$ We refer to $\nabla
f$ as the optimal map pushing $\mu$ to $\nu.$ Detailed proofs
of these results can be found in \cite[Theorems 9.4 and
10.41]{Vil08}.

We shall present four applications to gradient flows in the metric space $(\calP_2(\R^d), W_2).$

 \subsection*{The functionals}

In all of our examples below we shall consider a convex potential $V \in C^2(\R^d;\R)$ satisfying the (strong) assumption that the Hessian $D^2 V$ is bounded, or equivalently since $V$ is convex, that there exists $c \geq 0$ such that
\begin{align} \label{eq:Delta-bound}
 \Delta V(x) \leq c, \qquad x \in \R^d.
\end{align}
The \emph{potential energy} $\V : \calP_2(\R^d) \to \R$ given by
\begin{align*}
 \V(\mu) &:= \int_{\R^d} V(x) \,
d\mu(x),
\end{align*}
is well-defined, since the assumption on the Hessian implies that $|V|$ is of at most quadratic growth.

We shall also consider the (negative of the) \emph{Boltzmann entropy} $\H : \calP_2(\R^d) \to \R \cup\{ + \infty\}$ given by
\begin{align*}
 \H(\mu) &:= \left\{ \begin{array}{ll}
   \int_{\R^d} \rho(x) \log \rho(x)\, dx,
    &\mu = \rho \mathscr{L}^d, \\
 + \infty, & \text{otherwise},
\end{array} \right.
\end{align*}
and the \emph{R\'{e}nyi entropy} $\F : \calP_2(\R^d) \to \R \cup\{ + \infty\}$ defined for $m >1,$ by 
 \beq\bal \label{eq:Renyi}
 \F(\mu) &:= \left\{ \begin{array}{ll}
   \frac{1}{m-1} \int_{\R^d} \rho^m(x) \, dx,
     &\mu = \rho \mathscr{L}^d, \\
 + \infty, & \text{otherwise}.
\end{array} \right.
 \eal\eeq
 
A famous result by McCann \cite{McC97} asserts that the functionals $\V,$ $\H,$ and $\F$ are displacement convex, that is, convex along geodesics in $\calP_2(\R^d).$ 
In the first part of \cite{AGS}, an abstract theory of gradient flows in metric spaces has been developed for functionals which are convex along interpolating curves, not necessarily geodesics, along which the squared distance function satisfies an appropriate convexity condition. This condition fails for $W_2$-geodesics, but it holds for a different class of interpolating curves along which the functionals $\V,$ $\H,$ and $\F$ are convex as well \cite[Propositions 9.3.2 and 9.3.9]{AGS}. 
For our purpose, it is important to note that as a consequence (see \cite[Theorem 4.1.2]{AGS}), $(A_1)$ is satisfied for any pair of functionals chosen from $\V,$ $\H,$ and $\F$. 

Moreover, the first inclusion of $(A_2)$ is satisfied in this situation, since
\begin{align*}
 \Dom(\V) = \ov{\Dom(\H)} = \ov{\Dom(\F)} = \calP_2(\R^d).
\end{align*} 
 
In order to prove that the second inclusion of $(A_2)$ and $(A_3)$ hold in the examples below, we shall use the following known result which provides formulas for the densities of the resolvents. We let $J_h^\phi$ denote the resolvent associated with a functional $\phi$ defined on $\R^d$ or $\calP_2(\R^d).$

 \blem \label{lem:pushForwardDensity}
Let $\mu \in \calP_2(\R^d)$ and $h > 0.$ 
\ben
\item \label{item:pushForwardDensity-1} 
We have $J_h^\V \mu = (J_h^V)_\# \mu.$
If $\mu = \rho \mathscr{L}^d,$ then $J_h^\V \mu = \rho^1 \mathscr{L}^d,$ where
 \begin{align} \label{eq:density}
 \rho^1(x) = \rho( x + h \nabla V(x)) \det ( I + h D^2 V(x) ),
   \qquad x \in \R^d,
 \end{align}

\item \label{item:pushForwardDensity-2}
We have $J_h^\H \mu = \rho^1
\mathscr{L}^d$ for some $\rho^1 \in W^{1,1}(\R^d).$ Let $T$ be the
optimal map pushing $J_h^\H \mu$ to $\mu.$ Then we have
 \begin{align} \label{eq:Tudorascu}
 h  \nabla \rho^1(x) = (T(x) - x)\rho^1(x),
  \quad
\mathscr{L}^d\text{-a.e.}
 \end{align}

\item \label{item:pushForwardDensity-3}
We have $J_h^\F \mu = \rho^1
\mathscr{L}^d$ for some $\rho^1 \in W^{1,1}(\R^d).$ Let $T$ be the
optimal map pushing $J_h^\F \mu$ to $\mu.$ Then we have
 \begin{align} \label{eq:Tudorascu-ext}
 h  \nabla \big(\rho^1\big)^m(x) = (T(x) - x)\rho^1(x),
  \quad
\mathscr{L}^d\text{-a.e.}
 \end{align}
\een 
 \elem

 \bpf
The result follows from \cite[Lemma 10.1.2. and Theorem 10.4.6]{AGS}; see also \cite[Proposition 3]{Tud08}. For the convenience of the reader we provide a simple direct proof of \eqref{item:pushForwardDensity-1}.
 
For $x \in \R^d$ we use the fact that $J_h^V(x) = (I + h \nabla
V)^{-1}(x)$ to obtain
 \begin{align*}
  V(J_h^V(x)) + \frac{1}{2h} |x - J_h^V(x)|^2
      \leq  V(y) + \frac{1}{2h} |x - y|^2,
 \end{align*}
for any $y \in \R^d.$
Take $\nu \in \calP_2(\R^d)$ and let $T$ be the optimal map pushing $\mu$
to $\nu.$ Using the estimate above we obtain
 \begin{align*}
   \V((J_h^V)_\# \mu)  + \frac{1}{2h} W_2^2( \mu, (J_h^V)_\#\mu )
 &    \leq \int_{\R^d}  V(J_h^V(x))
       + \frac{1}{2h}  |x-J_h^V(x)|^2 \, d\mu(x)
  \\&    \leq \int_{\R^d}  V(T(x))
       + \frac{1}{2h}  |x-T(x)|^2 \, d\mu(x)
  \\& =  \V(\nu) + \frac{1}{2h} W_2^2( \mu, \nu),
 \end{align*}
which implies that $J_h^\V \mu  = (J_h^V)_\# \mu.$ This proves the first assertion.

To complete the proof, we note that,  $(I + h \nabla V)_\#(J_h^\V \mu) =  \mu.$
The convexity of $V$ implies that $\det(I + h D^2 V(x)) >0$ for all $x\in \R^d,$ and therefore \eqref{eq:density} follows from the change of variable formula.
 \epf

  \subsection*{Compatibility of the functionals}

It has been shown in \cite{JKO98} that the $L^2$-Wasserstein gradient flow associated with the sum $\H + \V$ solves the Fokker-Planck equation
\begin{align*}
 \partial_t \rho = \Delta \rho + \nabla \cdot (\rho \nabla V)
\end{align*}
in an appropriate sense.
Similarly, in \cite{O01} it has been shown that the gradient flow associated with $\F$ is a solution to the porous medium equation 
\begin{align*}
\partial_t \rho = \Delta \rho^m.
\end{align*}

The following result shows that the assumptions of Proposition \ref{prop:suff-for-A4} are satisfied in several examples.

 \bpr \label{prop:compatible}
Let $h > 0$ and let $c \geq 0$ be as in \eqref{eq:Delta-bound}.
 \ben
\item \label{item:VH}
For $\mu \in \Dom(\H)$ we have
 \begin{align} \label{eq:H-est}
  \H(J_h^\V\mu)  \leq \H(\mu) + ch.
 \end{align}
\item \label{item:HV}
For $\mu \in \calP_2(\R^d)$ we have
 \begin{align} \label{eq:V-est}
  \V(J_h^\H \mu)  \leq \V(\mu)   + ch.
 \end{align}
\item \label{item:pm-one}
For $\mu \in \Dom(\F)$ we have
\begin{align} \label{eq:F-est}
   \F(J_h^\V \mu)
		 \leq e^{(m-1)ch} \F(\mu).
\end{align}
\item \label{prop:pm-two}
For $\mu \in \calP_2(\R^d)$ we have
\begin{align} \label{eq:VF-est}
\V(J_h^\F \mu) \leq \V(\mu) + c(m-1)h \F(J_h^\F \mu).
\end{align}
 \een
 \epr

 \bpf
$(1)$ 
We write $\mu^1 := J_h^\V\mu$ and denote the densities of $\mu$ and
$\mu^1$ by $\rho$ and $\rho^1$ respectively. Using the convexity of $V,$ the inequality $\log\det(I+A) \leq \tr(A)$ which holds for any non-negative symmetric matrix $A,$ and the assumption on $\Delta V,$ we obtain
 \begin{align} \label{eq:pointwise}
  0 \leq \log \det (I + h D^2 V(x))
    \leq h \Delta V(x)
    \leq ch, \qquad x \in \R^d.
 \end{align}
Lemma \ref{lem:pushForwardDensity} implies that
 \begin{align*}
  \H(\mu) & +  \int_{\R^d} \log \det ( I + h D^2 V(x) ) \,
 d\mu^1(x)
 \\ & = \int_{\R^d} \log \rho(x) \, d\mu(x)
        +   \int_{\R^d} \log \det ( I + h D^2 V(x) )   \, d\mu^1(x)
 \\ & = \int_{\R^d} \log \rho(x + h \nabla V(x)) \, d\mu^1(x)
        +   \int_{\R^d} \log \det ( I + h D^2 V(x) )   \, d\mu^1(x)
 \\ & =  \int_{\R^d} \log\rho^1(x)   \, d\mu^1(x)
 \\ & =  \H( \mu^1).
 \end{align*}
The conclusion follows by combining this identity with \eqref{eq:pointwise}.

$(2)$ 
Let $V^+$ and $V^-$ denote the positive and negative part of $V$ respectively. Since $V$ is convex, there exist $k_1, k_2 \geq 0$ such that $|V^-(x)| \leq k_1(1 + k_2 |x|)$ for all $x \in \R^d.$ Since $\mu^1 := J_h^\H \mu \in \calP_2(\R^d),$ this implies that $V^- \in L^1(\R^d;\mu^1).$

To show that $V^+ \in L^1(\R^d;\mu^1)$ as well, let $T$ denote the optimal map pushing $\mu^1 := J_h^\H \mu$ to
$\mu,$ and note that by the convexity of $V,$
 \begin{align*}
 V^+(x) = V(x) + V^-(x)
  \leq  V(T(x)) + \ip{\nabla V(x), x - T(x)} + V^-(x).
 \end{align*}
We claim that the first two summands at the right-hand side are contained in $L^1(\R^d,\mu^1).$

Indeed, since $\mu \in \Dom(\V)$ and $\mu = T_\# \mu^1$ we have $V \circ T \in L^1(\mu^1)$ and
 \begin{align*}
   \int_{\R^d} V(T(x)) \, d\mu^1(x) = \int_{\R^d} V(x) \, d\mu(x).
 \end{align*}
Furthermore, we note that the convexity of $V$ and the upper bound on its Laplacian imply that the Hessian of $V$ is bounded. As a consequence,  $\nabla V$ is of at most linear growth, which in view of the fact that $\mu^1 \in \calP_2(\R^d)$ implies that $\nabla V \in L^2(\R^d,\mu^1;\R^d).$ Since $\|I - T \|_{L^2(\R^d, \mu; \R^d)} = W_2(\mu,\mu^1) < \infty,$ it follows that $\ip{\nabla V, I - T} \in L^1(\R^d,\mu^1)$ by the Cauchy-Schwarz inequality. This proves the claim and we conclude that $V^+ \in L^1(\R^d;\mu^1).$

Lemma \ref{lem:pushForwardDensity}(\ref{item:pushForwardDensity-2}) implies that $\mu^1 = \rho^1
\mathscr{L}^d$ for some $\rho^1 \in W^{1,1}(\R^d).$
Using the
convexity of $V$ and \eqref{eq:Tudorascu},
 \beq\bal \label{eq:nablaInt}
 \V(\mu^1)
    & = \int_{\R^d} V(x) \, d\mu^1(x)
 \\ & \leq \int_{\R^d} V(T(x)) - \ip{\nabla V(x),T(x) - x}
            \, d\mu^1(x)
 \\ &    = \V(\mu)
               - h \int_{\R^d} \ip{\nabla V(x), \nabla \rho^1(x)}
                 \, dx.
 \eal\eeq

We shall show that
 \beq\bal  \label{eq:GaussAgain}
 \int_{\R^d} \big\langle \nabla V(x),
 \nabla\rho^1(x) \big\rangle   \, d x
 = - \int_{\R^d} \Delta V(x) \rho^1(x) \, dx.
 \eal\eeq
Then \eqref{eq:V-est} follows by combining
\eqref{eq:nablaInt}, \eqref{eq:GaussAgain}, and the condition
$\Delta V \leq c$.

To prove \eqref{eq:GaussAgain}, we set $n(y) = \frac{y}{|y|}$
for $y \neq 0,$ and apply the Gauss-Green theorem in the ball $B_R := \{ x \in \R^d : |x| < R\}$ to write for
$R > 0,$
 \beq\bal \label{eq:Gauss}
 \int_{\partial B_R} \rho^1(y) \big\langle \nabla V(y),
 n(y) \big\rangle   \, d S(y) & =
 \int_{B_R} \big\langle \nabla V(x),
 \nabla\rho^1(x) \big\rangle   \, d x
\\ & \qquad
+ \int_{B_R} \Delta V(x) \rho^1(x) \, dx.
 \eal\eeq
Now we observe that
 \begin{align*}
  \ip{\nabla V,\nabla \rho^1} = \frac{1}{h}\ip{\nabla V,T - I}\rho^1 \in L^1(\R^d),
 \end{align*}
and $ \rho^1 \Delta V \in L^1(\R^d)$ as a consequence of the
assumption that $0 \leq \Delta V \leq c.$ Therefore the
dominated convergence theorem implies that the right hand side
of \eqref{eq:Gauss} converges as $R \to \infty$. In particular it follows
that
 \begin{align} \label{eq:theOne}
   L := \lim_{R \to \infty}
     \int_{\partial B_R} \rho^1(y) \big\langle \nabla V(y),
 n(y) \big\rangle   \, d S(y)
 \end{align}
exists. On the other hand, since we already showed that $\nabla V \in L^2(\R, \mu^1;\R^d),$ it follows that
$\rho^1 \ip{\nabla V,n} \in L^1(\R^d)$, hence the coarea
formula (see, e.g., \cite[Proposition 1, p.118]{EG92}) implies
that for a.e. $R
>0,$
 \begin{align} \label{eq:theOther}
     \int_{\partial B_R} \rho^1(y) \big\langle \nabla V(y),
 n(y) \big\rangle   \, d S(y)
     = \frac{d}{dR} \bigg(
      \int_{B_R}  \rho^1(x) \big\langle \nabla V(x),
 n(x) \big\rangle  \, dx \bigg).
 \end{align}
Combining \eqref{eq:theOne} and \eqref{eq:theOther} we
conclude (as in e.g. \cite[Proof of Theorem 1]{Tud08})  that
$L = 0,$ and therefore \eqref{eq:GaussAgain} follows from
\eqref{eq:Gauss}.
 
(3) 
Write $\mu^1 := J_h^\V \mu$ and $\mu^1 = \rho^1 \mathscr{L}^d.$ Using the convention that $t^{m-1} = 0$ if $t = 0,$ it follows from Lemma \ref{lem:pushForwardDensity} that
\beq\bal\label{eq:F-formula}
 \F(\mu^1)
    & = \frac{1}{m-1} \int_{ \R^d} (\rho^1(x))^{m-1} \, d\mu^1(x)
  \\& = \frac{1}{m-1} \int_{\R^d}
  \Big(\rho(x + h \nabla V(x))
              \det(I + h D^2 V(x))\Big)^{m-1} \, d\mu^1(x)
  \\& = \frac{1}{m-1} \int_{\R^d} \Big(\rho(y)
              \det\big(I + h D^2 V( J_h^V(y) ) \big)\Big)^{m-1}
                 \, d\mu(y).
\eal\eeq
Using the inequality $\det(I+A) \leq e^{\tr(A)},$ which holds for any non-negative symmetric matrix $A,$ we obtain
\begin{align}
   \F(\mu^1)
\leq \frac{e^{(m-1)ch}}{m-1} \int_{\R^d} \big(\rho(y) \big)^{m-1}
               \, d\mu(y)
 = e^{(m-1)ch} \F(\mu),
\end{align}
which proves \eqref{eq:F-est}.

(4) 
As the proof is very similar to the proof of (2), we will only give a sketch of the argument.

Set $\mu^1 := J_h^\F \mu.$ Arguing as in the proof of (2), we infer that $V^-$ and $V^+$ are contained in $L^1(\R^d;\mu^1).$
Furthermore, Lemma \ref{lem:pushForwardDensity}(\ref{item:pushForwardDensity-3}) implies that $\mu^1 = \rho^1
\mathscr{L}^d$ for some $\rho^1 \in W^{1,1}(\R^d).$
Using the convexity of $V$ and \eqref{eq:Tudorascu-ext},
 \begin{align*}
  \V(\mu^1)
    & = \int_{\R^d} V(x) \, d\mu^1(x)
 \\ & \leq \int_{\R^d} V(T(x)) - \ip{\nabla V(x),T(x) - x}
            \, d\mu^1(x)
 \\ &    = \V(\mu)
               - h \int_{\R^d} \ip{\nabla V(x), \nabla (\rho^1)^m(x)}
                 \, dx.
 \end{align*}
As in the proof of \eqref{item:HV}, it follows that
 \beq\bal  \label{eq:GaussAgain-again}
 \int_{\R^d} \big\langle \nabla V(x),
 \nabla(\rho^1)^m(x) \big\rangle   \, d x
 = - \int_{\R^d} \Delta V(x) \rho^1(x)^m \, dx,
 \eal\eeq
and consequently,
\begin{align*}
 \V(\mu^1)
   & \leq \V(\mu) + h \int_{\R^d} \Delta V(x) (\rho^1(x))^m \, dx
 \\& \leq \V(\mu) + ch\int_{\R^d} (\rho^1(x))^m \, dx
    =    \V(\mu) + c(m-1) h \F(\mu^1).
\end{align*}
 \epf

 \begin{remark}
The estimates \eqref{eq:nablaInt}--\eqref{eq:GaussAgain}
extend a recent result by Tudorascu \cite[Theorem 1]{Tud08}, who
considered the special case $V(x) = \frac12 |x|^2.$
 \end{remark}

\begin{theorem}\label{thm:trotter-appl}
In each of the following four cases, the functionals $\phi^1$ and $\phi^2$ satisfy $(A_1),$ $(A_2),$ and $(A_3)$:
\ben[]
 \item (1)    $\phi^1 = \H$ and $\phi^2 = \V,$ \qquad \qquad 
 		(3)	  $\phi^1 = \F$ and $\phi^2 = \V,$
 \item 	 (2)  $\phi^1 = \V$ and $\phi^2 = \H,$ \qquad \qquad 
 		  (4) $\phi^1 = \V$ and $\phi^2 = \F.$
\een
As a consequence, the Trotter product formula from Theorem \ref{thm:trotter-formula} holds.
\end{theorem}

\begin{proof} 
In case (1), Proposition \ref{prop:compatible}(1) implies $(A_2)(ii)$ and Assumption $(1)$ of Proposition  \ref{prop:suff-for-A4}, hence $(A_3).$

In case (2), Proposition \ref{prop:compatible}(2) implies $(A_2)(ii)$ and Assumption $(1)$ of Proposition  \ref{prop:suff-for-A4}, hence $(A_3).$

In case (3), Proposition \ref{prop:compatible}(3) implies $(A_2)(ii)$ and Assumption $(2)$ of Proposition  \ref{prop:suff-for-A4}, hence $(A_3).$

In case (4), Proposition \ref{prop:compatible}(3 \& 4) implies $(A_2)(ii)$ and Assumption $(3)$ of Proposition  \ref{prop:suff-for-A4}, hence $(A_3).$
\end{proof}

 \bibliographystyle{ams-pln}
\bibliography{trotter}

 \end{document}